\documentclass[11pt]{amsart}
\author[C.~Sanna]{Carlo Sanna$^\dagger$}
\address{\parbox{\linewidth}{
		Department of Mathematical Sciences, Politecnico di Torino\\
		Corso Duca degli Abruzzi 24, 10129 Torino, Italy\\[-8pt]}}
\email{carlo.sanna@polito.it}
\thanks{$\dagger$ C.~Sanna is a member of the INdAM group GNSAGA and of CrypTO, the group of Cryptography and Number Theory of the Politecnico di Torino}

\title[Inverse of a Fibonacci modulo a Fibonacci being a Fibonacci]{On the inverse of a Fibonacci number modulo a Fibonacci number being a Fibonacci number}

\usepackage{amsmath}
\usepackage{amssymb}
\usepackage{amsthm}
\usepackage{geometry}
\usepackage{color}
\usepackage{hyperref}
\hypersetup{colorlinks=true}
\usepackage{enumitem}
\setlist[enumerate]{label=(\roman*),labelindent=1em,itemsep=0.5em,topsep=0.5em}

\newtheorem{theorem}{Theorem}[section]

\newtheorem{lemma}[theorem]{Lemma}
\theoremstyle{remark}

\begin{document}

\maketitle

\begin{abstract}
	Let $(F_n)_{n \geq 1}$ be the sequence of Fibonacci numbers.
	For all integers $a$ and $b \geq 1$ with $\gcd(a, b) = 1$, let $[a^{-1} \!\bmod b]$ be the multiplicative inverse of $a$ modulo $b$, which we pick in the usual set of representatives $\{0, 1, \dots, b-1\}$.
	Put also $[a^{-1} \!\bmod b] := \infty$ when $\gcd(a, b) > 1$.

	We determine all positive integers $m$ and $n$ such that $[F_m^{-1} \bmod F_n]$ is a Fibonacci number.
	This extends a previous result of Prempreesuk, Noppakaew, and Pongsriiam, who considered the special case $m \in \{3, n - 3, n - 2, n - 1\}$ and $n \geq 7$.

	Let $(L_n)_{n \geq 1}$ be the sequence of Lucas numbers.
	We also determine all positive integers $m$ and $n$ such that $[L_m^{-1} \bmod L_n]$ is a Lucas number.
\end{abstract}

\section{Introduction}

Let $(F_n)_{n \geq 1}$ be the sequence of Fibonacci numbers, which is defined by $F_1 := 1$, $F_2 := 1$, and $F_n := F_{n - 1} + F_{n - 2}$ for every integer $n \geq 3$.
Several authors studied modular multiplicative inverses related to Fibonacci numbers.
For instance, Komatsu, Luca, and Tachiya~\cite{MR3055682} (see also~\cite{MR3459557}) studied the multiplicative order of $F_{n+1}F_n^{-1}$ modulo $F_m$, where $m$ and $n$ are positive integers such that $\gcd(F_m, F_{n+1}F_n) = 1$.
Luca, St\u{a}nic\u{a}, and Yal\c{c}iner~\cite{MR3246778} studied the positive integers $M$ such that the invertible residue classes modulo $M$ represented by Fibonacci numbers form a subgroup.
For all integers $a$ and $b \geq 1$ with $\gcd(a, b) = 1$, let $[a^{-1} \bmod b]$ be the unique $x \in \{0,1, \dots, b-1\}$ such that $ax \equiv 1 \pmod b$.
For the sake of convenience, put also $[a^{-1} \bmod b] := \infty$ when $\gcd(a, b) > 1$.
Alecci, Murru, and Sanna~\cite{MR4574177} determined the Zeckendorf representation of $[a^{-1} \bmod F_n]$, for every fixed $a \geq 3$ and for every integer $n \geq 1$.
(The case $a = 2$ was previously solved by Prempreesuk, Noppakaew, and Pongsriiam~\cite{MR4057296}).
Motivated by some results in knot theory~\cite{MR3264672}, Song~\cite{zbMATH07142168} found four families of pairs $(F_m, F_n)$ of Fibonacci numbers such that $[F_m^{-1} \bmod F_n]$ and $[F_n^{-1} \bmod F_m]$ are both Fibonacci numbers.
Sanna~\cite{MR4538571} proved that these families, together with some isolated pairs, are indeed all the pairs of Fibonacci numbers with such a property.
For integers $n \geq 7$ and $m \in \{3, n-3,n-2,n-1\}$, Prempreesuk, Noppakaew, and Pongsriiam~\cite{MR4057296} found necessary and sufficient conditions for $[F_m^{-1} \bmod F_n]$ to be a Fibonacci number.

Our first result is the following.

\begin{theorem}\label{thm:inverse-fibonacci}
	Let $\ell, m, n $ be integers with $\ell \geq 2$, $m \geq 3$, $n \geq 4$ and $m < 4n$.
	Then
	\begin{equation}\label{equ:Fl-equal-to-inverse-Fm-mod-Fn}
		F_\ell = [F_m^{-1} \bmod F_n]
	\end{equation}
	if and only if:
	\begin{enumerate}[label=(c{\small\arabic*})]
		\item\label{ite:inverse-fibonacci1} $\ell = n - \tfrac1{2}\big(3 + (-1)^n\big)$ and $m = n - 2$; or
		\item\label{ite:inverse-fibonacci2} $\ell = n - \tfrac1{2}\big(3 - (-1)^n\big)$ and $m \in \{n - 1, n + 1, n + 2\}$; or
		\item\label{ite:inverse-fibonacci3} $\ell = 2$, $m = 2n - 2$, and $n$ is odd; or
		\item\label{ite:inverse-fibonacci4} $\ell = 2$, $m \in \{2n - 1, 2n + 1, 2n + 2\}$, and $n$ is even; or
		\item\label{ite:inverse-fibonacci5} $\ell = n - 2$ and $m = 3n - 2$; or
		\item\label{ite:inverse-fibonacci6} $\ell = n - 1$ and $m \in \{3n - 1, 3n + 1, 3n + 2\}$; or
		\item\label{ite:inverse-fibonacci7} $\ell = 2$ and $m = 4n - 1$.
	\end{enumerate}
\end{theorem}

We remark that, in Theorem~\ref{thm:inverse-fibonacci}, the conditions on $\ell, m, n$ are not restrictive.
In fact, since $F_1 = F_2 = 1$, there is no loss in generality in assuming that $\ell, m, n \geq 2$.
Moreover, the cases in which $m = 2$ or $n \in \{2, 3\}$ are easy to study.
Hence, for the sake of brevity, we did not include them.
Finally, in light of Lemma~\ref{lem:period-4n}, the condition $m < 4n$ is not a restriction.

Let $(L_n)_{n \geq 1}$ be the sequence of Lucas numbers, which is defined by $L_1 := 2$, $L_2 := 1$, and $L_n = L_{n - 1} + L_{n - 2}$ for every integer $n \geq 3$.

Our second result is the following.

\begin{theorem}\label{thm:inverse-lucas}
	Let $\ell, m, n$ be integers with $\ell \geq 1$, $m \geq 2$, $n \geq 5$, and $m \leq 4n$.
	Then
	\begin{equation}\label{equ:Ll-equal-to-inverse-Lm-mod-Ln}
		L_\ell = [L_m^{-1} \bmod L_n]
	\end{equation}
	if and only if:
	\begin{enumerate}[label=(d{\small\arabic*})]
		\item\label{ite:inverse-lucas1} $\ell = \tfrac1{2}(n \pm 1)$, $m = \tfrac1{2}(n \mp 1)$, and $n \equiv 1 \pmod 4$; or
		\item\label{ite:inverse-lucas2} $\ell = \tfrac1{2}(n+1)$, $m = \tfrac1{2}(3n+1)$, and $n$ is odd; or
		\item\label{ite:inverse-lucas3} $\ell = 1$ and $m = 2n - (-1)^n$; or
		\item\label{ite:inverse-lucas4} $\ell = \tfrac1{2}(n - 1)$, $m = \tfrac1{2}(5n + 1)$, and $n \equiv 1 \pmod 4$; or
		\item\label{ite:inverse-lucas5} $\ell = \tfrac1{2}(n + 1)$, $m = \tfrac1{2}(5n - 1)$, and $n \equiv 1 \pmod 4$; or
		\item\label{ite:inverse-lucas6} $\ell = \tfrac1{2}(n + 1)$, $m = \tfrac1{2}(7n + 1)$, and $n$ is odd.
	\end{enumerate}
\end{theorem}

We remark that, in Theorem~\ref{thm:inverse-lucas}, the conditions on $\ell, m, n$ are not restrictive.
In fact, the cases in which $m = 1$ or $n \in \{2, 3, 4\}$ are easy to study.
Hence, for the sake of brevity, we did not include them.
Furthermore, in light of Lemma~\ref{lem:period-4n}, the condition $m \leq 4n$ is not a restriction.

The proof of Theorem~\ref{thm:inverse-fibonacci} (resp.~\ref{thm:inverse-lucas}) is based on approximating $F_\ell F_m / F_n$ (resp.~$L_\ell L_m / L_n$) with an appropriate linear combination of Fibonacci (resp.~Lucas) numbers, which depends on the size of $m$ relatively to $\ell$ and $n$.
This approximation is sufficiently accurate to imply certain identities for the integers $\ell,m,n$ that satisfy~\eqref{equ:Fl-equal-to-inverse-Fm-mod-Fn} (resp.~\eqref{equ:Ll-equal-to-inverse-Lm-mod-Ln}).
Then, these identities make possible to determine $\ell, m,n$.

\section{Preliminaries}

It is well known that the Binet formulas
\begin{equation}\label{equ:binet}
	F_n = \frac{\alpha^n - \beta^n}{\alpha - \beta} \quad\text{ and }\quad L_n = \alpha^n + \beta^n
\end{equation}
hold for every integer $n \geq 1$, where $\alpha := (1 + \sqrt{5}) / 2$ and $\beta := (1 - \sqrt{5}) / 2$ (see, e.g., \cite[Ch.~1]{MR1761897}).

In fact, it is useful to extend the sequences of Fibonacci and Lucas numbers to all integers by using~\eqref{equ:binet}.
In particular, from \eqref{equ:binet} it follows that
\begin{equation}\label{equ:negative}
	F_{-n} = (-1)^{n + 1} F_n \quad\text{ and }\quad L_{-n} = (-1)^n L_n ,
\end{equation}
for every integer $n$.

We also need the following lemmas.

\begin{lemma}\label{lem:products}
	We have that
	\begin{enumerate}
		\item\label{ite:product-La-Lb} $L_a L_b = L_{a + b} + (-1)^b L_{a - b}$;
		\item\label{ite:product-5Fa-Fb} $5F_a F_b = L_{a + b} - (-1)^b L_{a - b}$;
		\item\label{ite:product-Fa-Fb-long} $F_a F_b = F_{a + b} - F_{a + b - 2} - F_{a - 1} F_{b - 1}$;
	\end{enumerate}
	for all integers $a, b$.
\end{lemma}
\begin{proof}
	These identities follow easily from \eqref{equ:binet}.
\end{proof}

\begin{lemma}\label{lem:identities}
	We have that
	\begin{enumerate}
		\item\label{ite:FaFb-FcFd} $F_a F_b - F_c F_d = (-1)^{a+1} F_{c-a} F_{d - a}$;
		\item\label{ite:LaLb-LcLd} $L_a L_b - L_c L_d = (-1)^a 5 F_{c-a} F_{d - a}$;
		\item\label{ite:5FaFb-LcLd} $5F_a F_b - L_c L_d = (-1)^{a+1} L_{c - a} L_{d - a}$;
	\end{enumerate}
	for all integers $a,b,c,d$ with $a + b = c + d$.
\end{lemma}
\begin{proof}
	These identities can be proved either directly from~\eqref{equ:binet}, or by taking appropriate differences of pairs of the identities of Lemma~\ref{lem:products} and, eventually, employing~\eqref{equ:negative}.
\end{proof}

\begin{lemma}\label{lem:period-4n}
	We have that
	\begin{enumerate}
		\item $F_{a + 4n} \equiv F_a \pmod {F_n}$;
		\item $L_{a + 4n} \equiv L_a \pmod {L_n}$;
	\end{enumerate}
	for all integers $a$ and $n \geq 1$.
\end{lemma}
\begin{proof}
	Using~\eqref{equ:binet} one can verify that
	\begin{equation*}
		F_{a + 4n} - F_a = F_n L_n L_{a + 2n} \quad \text{ and } \quad L_{a + 4n} - L_a = 5 F_n F_{a + 2n} L_n ,
	\end{equation*}
	from which the claim follows.
\end{proof}

\begin{lemma}\label{lem:inequalities}
	Let $a, b, n$ be integers.
	We have that:
	\begin{enumerate}
		\item\label{ite:inequality-Fa-Fb} $|F_a F_b| < F_n - 1$ if $|a| + |b| \leq n$ and $n \geq 5$;
		\item\label{ite:inequality-5Fa-Fb} $|5F_a F_b| < L_n - 1$ if $|a| + |b| \leq n - 1$ and $n \geq 6$;
		\item\label{ite:inequality-La} $|L_a| < L_n - 1$ if $|a| \leq n - 1$ and $n \geq 4$;
		\item\label{ite:inequality-La-Lb} $|L_a L_b| < L_n - 1$ if $|a| + |b| \leq n - 1$, $\{|a|,|b|\} \notin \{0, n - 1\}$, and $n \geq 6$.
	\end{enumerate}
\end{lemma}
\begin{proof}
	By~\eqref{equ:negative}, we have that $|F_k| = F_{|k|}$ and $|L_k| = L_{|k|}$, for every integer $k$.
	Hence, throughout the proof, we can assume that $a, b \geq 0$.
    Moreover, by symmetry, we can assume that $a \geq b$.

	Let us prove~\ref{ite:inequality-Fa-Fb} and~\ref{ite:inequality-5Fa-Fb}.
	If $b = 0$ or $a + b \leq 4$, then \ref{ite:inequality-Fa-Fb} and \ref{ite:inequality-5Fa-Fb} follow easily.
	Hence, assume that $b \geq 1$ and $a + b \geq 5$.
	Then, by Lemma~\ref{lem:products}\ref{ite:product-Fa-Fb-long}, we have that
	\begin{equation*}
		F_a F_b = F_{a + b} - F_{a + b - 2} - F_{a - 1} F_{b - 1} \leq F_n - F_3 < F_n - 1 ,
	\end{equation*}
	whenever $a + b \leq n$, which proves~\ref{ite:inequality-Fa-Fb}.
	Furthermore, by Lemma~\ref{lem:products}\ref{ite:product-5Fa-Fb}, we have that
	\begin{equation*}
		5 F_a F_b = L_{a + b} - (-1)^b L_{a - b} \leq L_{a + b} + L_{a - b} \leq L_{n - 1} + L_{n - 3} = L_n - L_{n - 4} < L_n - 1 ,
	\end{equation*}
	whenever $a + b \leq n - 1$ and $n \geq 6$, which proves~\ref{ite:inequality-5Fa-Fb}.

	If $a \leq n - 1$ and $n \geq 4$, then
    \begin{equation*}
        L_a \leq \max\{2,L_{n - 1}\} = L_{n - 1} = L_n - L_{n - 2} < L_n - 1 ,
    \end{equation*}
    which proves~\ref{ite:inequality-La}.

    By Lemma~\ref{lem:products}\ref{ite:product-La-Lb}, we have that
    \begin{equation*}
        L_{a + b} = L_{a + b} + (-1)^b L_{a - b} \leq L_{a + b} + L_{a - b} .
    \end{equation*}
    Therefore, if $a + b \leq n - 2$ and $n \geq 5$, then
    \begin{equation*}
        L_{a + b} \leq 2 L_{n - 2} = L_n - L_{n - 3} < L_n - 1 .
    \end{equation*}
    Furthermore, if $a + b = n - 1$, $b \geq 1$, and $n \geq 6$, then
    \begin{equation*}
        L_{a + b} \leq L_{n - 1} + L_{n - 3} = L_n - L_{n - 4} < L_n - 1 .
    \end{equation*}
    Thus~\ref{ite:inequality-La-Lb} is proved.
\end{proof}

\section{Proof of Theorem~\ref{thm:inverse-fibonacci}}

With a bit of patience, one can check that Theorem~\ref{thm:inverse-fibonacci} holds for $n = 4$.
Hence, hereafter, assume that $n \geq 5$.

Suppose that~\eqref{equ:Fl-equal-to-inverse-Fm-mod-Fn} is satisfied.
Hence, we have that $F_\ell < F_n$ and
\begin{equation}\label{equ:Fl-Fm-A-Fn-1}
	F_\ell F_m - A F_n = 1 ,
\end{equation}
for some integer $A \geq 0$.
In particular, from $F_\ell < F_n$ it follows that $\ell < n$.
Moreover, since $m \geq 3$, we get that $A \geq 1$.
Consequently, we have that $F_\ell F_m \geq F_n + 1$.
Hence, by Lemma~\ref{lem:inequalities}\ref{ite:inequality-Fa-Fb}, we get that $m > n - \ell$.

Define the four disjoint intervals
\begin{equation*}
	I_1 := {(n - \ell, n + \ell]}, \quad I_2 := {(n + \ell, 3n - \ell]}, \quad I_3 := {(3n - \ell, 3n + \ell]}, \quad I_4 := {(3n + \ell, 4n)} .
\end{equation*}
By the previous considerations, we have that $m$ belongs to exactly one of such intervals.

For every integer $k$, put $F_k^+ := F_k$ if $k \geq 0$, and $F_k^+ := 0$ if $k < 0$.
Then, define
\begin{equation*}
	A_{\ell, m, n} := F_{\ell + m - n} - (-1)^\ell F_{-\ell + m - n}^+ + (-1)^n F_{\ell + m - 3n}^+ - (-1)^{\ell + n} F_{-\ell + m - 3n}^+
\end{equation*}
and
\begin{equation}\label{equ:def-Blmn}
	B_{\ell, m, n} := F_\ell F_m - A_{\ell, m, n} F_n .
\end{equation}
Let us prove that
\begin{equation}\label{equ:Blmn-cases}
	B_{\ell, m, n} =
	\begin{cases}
		(-1)^{\ell + 1} F_{m - n} F_{n - \ell} & \text{ if } m \in I_1 ; \\
		(-1)^n F_{\ell} F_{m - 2n} & \text{ if } m \in I_2 ; \\
		(-1)^{\ell+n+1} F_{m - 3n} F_{n - \ell} & \text{ if } m \in I_3 ; \\
		F_{\ell} F_{m - 4n} & \text{ if } m \in I_4 .
	\end{cases}
\end{equation}

By Lemma~\ref{lem:identities}\ref{ite:FaFb-FcFd}, we have that
\begin{equation}\label{equ:Blmn-I1}
	F_\ell F_m - F_{\ell + m - n} F_n = (-1)^{\ell + 1} F_{m - n} F_{n - \ell} .
\end{equation}
Hence, we get that $B_{\ell, m, n} = (-1)^{\ell + 1} F_{m - n} F_{n - \ell}$ for each $m \in I_1$.

By~\eqref{equ:Blmn-I1}, Lemma~\ref{lem:identities}\ref{ite:FaFb-FcFd}, and~\eqref{equ:negative}, we have that
\begin{align}\label{equ:Blmn-I2}
	F_\ell &F_m - \big(F_{\ell + m - n} - (-1)^\ell F_{-\ell + m - n}\big) F_n =
	(-1)^{\ell + 1} (F_{m - n} F_{n - \ell} - F_{-\ell + m - n} F_n) \\
	&= (-1)^{\ell + m + n} F_{-\ell} F_{2n - m} = (-1)^n F_{\ell} F_{m - 2n} \nonumber
\end{align}
Hence, we get that $B_{\ell, m, n} = (-1)^n F_{\ell} F_{m - 2n}$ for each $m \in I_2$.

By~\eqref{equ:Blmn-I2} and Lemma~\ref{lem:identities}\ref{ite:FaFb-FcFd}, we have that
\begin{align}\label{equ:Blmn-I3}
	F_\ell &F_m - \big(F_{\ell + m - n} - (-1)^\ell F_{-\ell + m - n} + (-1)^n F_{\ell + m - 3n}\big) F_n \\
	&= (-1)^n (F_{\ell} F_{m - 2n} - F_{\ell + m - 3n} F_n)
	= (-1)^{\ell+n+1} F_{m - 3n} F_{n - \ell} . \nonumber
\end{align}
Hence, we get that $B_{\ell, m, n} = (-1)^{\ell+n+1} F_{m - 3n} F_{n - \ell}$ for each $m \in I_3$.

Finally, by~\eqref{equ:Blmn-I3},~\eqref{equ:negative}, and Lemma~\ref{lem:identities}\ref{ite:FaFb-FcFd}, we have that
\begin{align*}
	F_\ell &F_m - \big(F_{\ell + m - n} - (-1)^\ell F_{-\ell + m - n} + (-1)^n F_{\ell + m - 3n} - (-1)^{\ell + n} F_{-\ell + m - 3n}\big) F_n \\
	&= (-1)^{\ell+n+1} (F_{m - 3n} F_{n - \ell} - F_{-\ell + m - 3n} F_n)
	= (-1)^{\ell+m} F_{-\ell} F_{4n - m} = F_{\ell} F_{m - 4n} . \nonumber
\end{align*}
Hence, we get that $B_{\ell, m, n} = F_{\ell} F_{m - 4n}$ for each $m \in I_4$.
The proof of~\eqref{equ:Blmn-cases} is complete.

At this point, considering the four cases in \eqref{equ:Blmn-cases}, one can easily check that $B_{\ell, m, n}$ is equal to~$\pm F_a F_b$, where $a$ and $b$ are integers (depending on $\ell, m, n$) such that $|a| + |b| \leq n$.

Therefore, from~\eqref{equ:def-Blmn} and Lemma~\ref{lem:inequalities}\ref{ite:inequality-Fa-Fb}, we get that
\begin{equation*}
	\left|\frac{F_\ell F_m}{F_n} - A_{\ell, m, n}\right| = \frac{|B_{\ell,m,n}|}{F_n} = \frac{|F_a F_b|}{F_n} < 1 - \frac1{F_n} .
\end{equation*}
Consequently, recalling \eqref{equ:Fl-Fm-A-Fn-1}, we have that
\begin{equation*}
	|A - A_{\ell, m, n}| \leq \left|A - \frac{F_\ell F_m}{F_n}\right| + \left|\frac{F_\ell F_m}{F_n} - A_{\ell, m, n}\right| < \frac1{F_n} + \left(1 - \frac1{F_n}\right) = 1 ,
\end{equation*}
which implies that $A = A_{\ell, m, n}$, since $A$ and $A_{\ell, m, n}$ are both integers.
Then, from \eqref{equ:Fl-Fm-A-Fn-1} and \eqref{equ:def-Blmn}, we get that $B_{\ell, m, n} = 1$.

Note that, for every integer $k$, we have that $|F_k| = 1$ if and only if $k \in \{-2, -1, 1, 2\}$.
In~particular, we have that $F_{-2} = -1$ and $F_{-1} = F_1 = F_2 = 1$.
Therefore, from $B_{\ell, m, n} = 1$ we can determine $\ell$ and $m$ in terms of $n$ in each of the four cases in~\eqref{equ:Blmn-cases}.

If $m \in I_1$, then $(-1)^{\ell + 1} F_{m - n} F_{n - \ell} = 1$.
Hence, we have that $m \in \{n - 2, n - 1, n + 1, n + 2\}$ and $\ell \in \{n - 2, n - 1\}$ (recall that $\ell < n$).
If $m = n - 2$, then either $\ell = n - 2$ and $n$ is even, or $\ell = n - 1$ and $n$ is odd.
This is case~\ref{ite:inverse-fibonacci1}.
If $m \in \{n - 1, n + 1, n + 2\}$, then either $\ell = n - 1$ and $n$ is even, or $\ell = n - 2$ and $n$ is odd.
This is case~\ref{ite:inverse-fibonacci2}.

If $m \in I_2$, then $(-1)^n F_\ell F_{m - 2n} = 1$.
Hence, $\ell = 2$ and $m \in \{2n-2, 2n-1, 2n+1, 2n+2\}$.
If $m = 2n - 2$, then $F_{m - 2n} = -1$ and consequently $n$ is odd, which is case~\ref{ite:inverse-fibonacci3}.
If $m \in \{2n - 1, 2n + 1, 2n + 2\}$, then $F_{m - 2n} = 1$ and consequently $n$ is even, which is case~\ref{ite:inverse-fibonacci4}.

If $m \in I_3$, then $(-1)^{\ell + n + 1} F_{m - 3n} F_{n - \ell} = 1$.
Hence, it follows that $\ell \in \{n - 2, n - 1\}$ and $m \in \{3n - 2, 3n - 1, 3n + 1, 3n + 2\}$.
If $\ell = n - 2$, then $(-1)^{\ell + n + 1} F_{n - \ell} = -1$ and consequently $m = 3n - 2$, which is case~\ref{ite:inverse-fibonacci5}.
If $\ell = n - 1$, then $(-1)^{\ell + n + 1} F_{n - \ell} = 1$ and consequently $m \in \{3n - 1, 3n + 1, 3n + 2\}$, which is case~\ref{ite:inverse-fibonacci6}.

If $m \in I_4$, then $F_\ell F_{m - 4n} = 1$.
Hence, we have that $\ell = 2$ and $m = 4n - 1$ (recall that $m < 4n$), which is case~\ref{ite:inverse-fibonacci7}.

At this point, we have proved that if~\eqref{equ:Fl-equal-to-inverse-Fm-mod-Fn} is true then the integers $\ell, m, n$ are of the form given by~\ref{ite:inverse-fibonacci1}--\ref{ite:inverse-fibonacci7}.

Vice versa, using~\eqref{equ:Blmn-cases}, on can easily verify that if $\ell, m, n$ are of the form given by~\ref{ite:inverse-fibonacci1}--\ref{ite:inverse-fibonacci7} then $B_{\ell, m, n} = 1$.
In turn, by~\eqref{equ:def-Blmn}, this implies that~\eqref{equ:Fl-equal-to-inverse-Fm-mod-Fn} is true.

The proof of Theorem~\ref{thm:inverse-fibonacci} is complete.

\section{Proof of Theorem~\ref{thm:inverse-lucas}}

With a bit of patience, one can check that Theorem~\ref{thm:inverse-lucas} holds for $n = 5$.
Hence, hereafter, assume that $n \geq 6$.

Suppose that~\eqref{equ:Ll-equal-to-inverse-Lm-mod-Ln} is satisfied.
Hence, we have that $L_\ell < L_n$ and
\begin{equation}\label{equ:Ll-Lm-C-Ln-1}
	L_\ell L_m - C L_n = 1 ,
\end{equation}
for some integer $C \geq 0$.
In particular, from $L_\ell < L_n$ it follows that $\ell < n$.
Moreover, since $m \geq 2$, we get that $C \geq 1$.
Consequently, we have that $L_\ell L_m \geq L_n + 1$.
Hence, by Lemma~\ref{lem:inequalities}\ref{ite:inequality-La-Lb}, we get that $m \geq n - \ell$.

Define the four disjoint intervals
\begin{equation*}
	J_1 := {[n - \ell, n + \ell)}, \quad J_2 := {[n + \ell, 3n - \ell)}, \quad
	J_3 := {[3n - \ell, 3n + \ell)}, \quad J_4 := {[3n + \ell, 4n]} .
\end{equation*}
By the previous considerations, we have that $m$ belongs to exactly one of such intervals.

For every integer $k$, put $L_k^+ := L_k$ if $k \geq 0$, and $L_k^+ := 0$ if $k < 0$.
Then, define
\begin{equation*}
	C_{\ell, m, n} := L_{\ell + m - n} + (-1)^\ell L_{-\ell + m - n}^+ - (-1)^n L_{\ell + m - 3n}^+ - (-1)^{\ell + n} L_{-\ell + m - 3n}^+
\end{equation*}
and
\begin{equation}\label{equ:def-Dlmn}
	D_{\ell, m, n} := L_\ell L_m - C_{\ell, m, n} L_n .
\end{equation}
Let us prove that
\begin{equation}\label{equ:Dlmn-cases}
	D_{\ell, m, n} =
		\begin{cases}
			(-1)^\ell 5 F_{m - n} F_{n - \ell} & \text{ if } m \in J_1 ; \\
			(-1)^{n + 1} L_\ell L_{m - 2n} & \text{ if } m \in J_2 ; \\
			(-1)^{\ell + n + 1} 5 F_{m - 3n} F_{n - \ell} & \text{ if } m \in J_3 ; \\
			L_\ell L_{m - 4n} & \text{ if } m \in J_4 .
		\end{cases}
\end{equation}

From Lemma~\ref{lem:identities}\ref{ite:LaLb-LcLd}, it follows that
\begin{equation}\label{equ:D-J1}
	L_\ell L_m - L_{\ell + m - n} L_n = (-1)^\ell 5 F_{m - n} F_{n - \ell} .
\end{equation}
Hence, we have that $D_{\ell, m, n} = (-1)^\ell 5 F_{m - n} F_{n - \ell}$ for every $m \in J_1$.

From~\eqref{equ:D-J1}, Lemma~\ref{lem:identities}\ref{ite:5FaFb-LcLd}, and~\eqref{equ:negative}, it follows that
\begin{align}\label{equ:D-J2}
	L_\ell& L_m - \big(L_{\ell + m - n} + (-1)^\ell L_{-\ell + m - n}\big) L_n = (-1)^\ell (5 F_{m - n} F_{n - \ell} - L_{-\ell + m - n} L_n) \\
	&= (-1)^{\ell + m + n + 1} L_{-\ell} L_{2n - m} = (-1)^{n + 1} L_{\ell} L_{m - 2n} . \nonumber
\end{align}
Hence, we have that $D_{\ell, m, n} = (-1)^{n + 1} L_{\ell} L_{m - 2n}$ for every $m \in J_2$.

From~\eqref{equ:D-J2} and Lemma~\ref{lem:identities}\ref{ite:LaLb-LcLd}, it follows that
\begin{align}\label{equ:D-J3}
	L_\ell& L_m - \big(L_{\ell + m - n} + (-1)^\ell L_{-\ell + m - n}  - (-1)^n L_{\ell + m - 3n}\big) L_n \\
	&= (-1)^{n+1}(L_{\ell} L_{m - 2n} - L_{\ell + m - 3n} L_n) = (-1)^{\ell+n+1} 5 F_{m - 3n} F_{n - \ell} . \nonumber
\end{align}
Hence, we have that $D_{\ell, m, n} = (-1)^{\ell+n+1} 5 F_{m - 3n} F_{n - \ell}$ for every $m \in J_3$.

Finally, from ~\eqref{equ:D-J3}, Lemma~\ref{lem:identities}\ref{ite:5FaFb-LcLd}, and~\eqref{equ:negative}, it follows that
\begin{align*}
	L_\ell& L_m - \big(L_{\ell + m - n} + (-1)^\ell L_{-\ell + m - n}  - (-1)^n L_{\ell + m - 3n} - (-1)^{\ell + n} L_{-\ell + m - 3n}\big) L_n \\
	&= (-1)^{\ell+n+1}(5 F_{m - 3n} F_{n - \ell} - L_{-\ell + m - 3n} L_n) =
	(-1)^{\ell+m} L_{-\ell} L_{4n - m} = L_{\ell} L_{m - 4n} . \nonumber
\end{align*}
Hence, we have that $D_{\ell, m, n} = L_{\ell} L_{m - 4n}$ for every $m \in J_4$.
The proof of \eqref{equ:Dlmn-cases} is complete.

Suppose that $C^\prime, D^\prime$ are integers such that
\begin{equation*}
    \left|\frac{L_\ell L_m}{L_n} - C^\prime\right| < 1 - \frac1{L_n} \quad\text{ and }\quad D^\prime = L_\ell L_m - C^\prime L_n .
\end{equation*}
Then, by \eqref{equ:Ll-Lm-C-Ln-1}, we have that
\begin{equation*}
    |C - C^\prime| \leq \left|C - \frac{L_\ell L_m}{L_n}\right| + \left|\frac{L_\ell L_m}{L_n} - C^\prime\right| < \frac1{L_n} + \left(1 - \frac1{L_n}\right) = 1 .
\end{equation*}
Consequently, we have that $C^\prime = C$ and, by \eqref{equ:Ll-Lm-C-Ln-1} again, that $D^\prime = 1$.

Hereafter, we will make use of such a fact several times, by taking $C^\prime = C_{\ell,m,n} + s$ and $D^\prime = D_{\ell,m,n} - s L_n$ for some $s \in \{-1,0,1\}$.

We will also use the fact that, for every integer $k$, we have that $|L_k| = 1$ if and only if $k \in \{-1, 1\}$.
Precisely, we have that $L_{-1} = -1$ and $L_1 = 1$.

If $m = n - \ell$, then from~\eqref{equ:Dlmn-cases}, \eqref{equ:negative}, and Lemma~\ref{lem:products}\ref{ite:product-5Fa-Fb} it follows that
\begin{equation*}
    D_{\ell,m,n} = (-1)^\ell 5 F_{-\ell} F_{n - \ell} = -5 F_\ell F_{n - \ell} = -L_n + (-1)^{n - \ell} L_{2\ell - n} .
\end{equation*}
Hence, by Lemma~\ref{lem:inequalities}\ref{ite:inequality-La}, we have that
\begin{equation*}
    \left|\frac{L_\ell L_m}{L_n} - (C_{\ell,m,n} - 1)\right| = \frac{|L_{2\ell - n}|}{L_n} < 1 - \frac1{L_n} ,
\end{equation*}
which implies that $(-1)^{n - \ell} L_{2\ell - n} = 1$.
Therefore, either $n - \ell$ is even and $2\ell - n = 1$, or $n - \ell$ is odd and $2\ell - n = -1$.
Recalling that $m = n - \ell$, it follows that $\ell = \tfrac1{2}(n \pm 1)$, $m = \tfrac1{2}(n \mp 1)$, and $n \equiv 1 \pmod 4$, which is case~\ref{ite:inverse-lucas1}.

If $m \in J_1 \setminus \{n - \ell\}$, then \eqref{equ:Dlmn-cases} and Lemma~\ref{lem:inequalities}\ref{ite:inequality-5Fa-Fb} yield that
\begin{equation*}
    \left|\frac{L_\ell L_m}{L_n} - C_{\ell,m,n}\right| = \frac{5|F_{m - n}F_{n - \ell}|}{L_n} < 1 - \frac1{L_n} ,
\end{equation*}
which implies that $(-1)^\ell 5 F_{m - n}F_{n - \ell} = 1$.
However, this last equality is clearly impossible.

If $m = n + \ell$, then \eqref{equ:Dlmn-cases} and Lemma~\ref{lem:products}\ref{ite:product-La-Lb} yield that
\begin{equation*}
    D_{\ell,m,n} = (-1)^{n + 1} L_{\ell} L_{\ell - n} = (-1)^{n + 1} L_{2\ell - n} - (-1)^\ell L_n .
\end{equation*}
Hence, by Lemma~\ref{lem:inequalities}\ref{ite:inequality-La}, we have that
\begin{equation*}
    \left|\frac{L_\ell L_m}{L_n} - \big(C_{\ell,m,n} - (-1)^\ell\big)\right| = \frac{|L_{2\ell - n}|}{L_n} < 1 - \frac1{L_n} ,
\end{equation*}
which implies that $(-1)^{n + 1} L_{2\ell - n} = 1$.
Therefore, either $n$ is odd and $2\ell - n = 1$, or $n$ is even and $2\ell - n = -1$.
However, the latter case is impossible.
Hence, recalling that $m = n + \ell$, we get that $\ell = \tfrac1{2}(n + 1)$, $m = \tfrac1{2}(3n + 1)$, and $n$ is odd, which is case~\ref{ite:inverse-lucas2}.

If $m \in J_2 \setminus \{n + \ell\}$ and $(\ell, m) \neq (n-1, 2n)$, then \eqref{equ:Dlmn-cases} and Lemma~\ref{lem:inequalities}\ref{ite:inequality-La-Lb} yield that
\begin{equation*}
    \left|\frac{L_\ell L_m}{L_n} - C_{\ell, m, n}\right| = \frac{|L_\ell L_{m - 2n}|}{L_n} < 1 - \frac1{L_n} ,
\end{equation*}
which implies that $(-1)^{n + 1} L_\ell L_{m - 2n} = 1$.
Thus $\ell = 1$ and either $m = 2n + 1$ and $n$ is odd, or $m = 2n - 1$ and $n$ is even.
This is case~\ref{ite:inverse-lucas3}.

If $\ell = n - 1$ and $m = 2n$ then, by~\eqref{equ:Dlmn-cases}, we get that
\begin{equation*}
	D_{\ell,m,n} = (-1)^{n + 1} 2 L_{n - 1} = (-1)^{n + 1} L_n + (-1)^{n + 1} L_{n - 3} .
\end{equation*}
Hence, by Lemma~\ref{lem:inequalities}\ref{ite:inequality-La}, we have that
\begin{equation*}
	\left|\frac{L_\ell L_m}{L_n} - \big(C_{\ell,m,n} + (-1)^{n+1}\big)\right| = \frac{|L_{n-3}|}{L_n} < 1 - \frac1{L_n} ,
\end{equation*}
which implies that $(-1)^{n + 1} L_{n - 3} = 1$, but this last equality is impossible.

If $m = 3n - \ell$, then \eqref{equ:Dlmn-cases}, \eqref{equ:negative}, and Lemma~\ref{lem:products}\ref{ite:product-5Fa-Fb} yield that
\begin{align*}
    D_{\ell, m, n} &= (-1)^{\ell + n + 1} 5F_{-\ell}F_{n - \ell} = (-1)^n 5F_{\ell}F_{n - \ell} = (-1)^n \big(L_n - (-1)^{n - \ell} L_{2\ell - n}\big) \\
    &= (-1)^n L_n - (-1)^{\ell} L_{2\ell - n} . \nonumber
\end{align*}
Hence, by Lemma~\ref{lem:inequalities}\ref{ite:inequality-La}, we have that
\begin{equation*}
    \left|\frac{L_\ell L_m}{L_n} - \big(C_{\ell, m, n} + (-1)^n\big)\right| = \frac{|L_{2\ell - n}|}{L_n} < 1 - \frac1{L_n} ,
\end{equation*}
which implies that $(-1)^{\ell + 1} L_{2\ell - n} = 1$.
Therefore, either $\ell$ is even and $2\ell - n = -1$, or $\ell$ is odd and $2\ell - n = 1$.
That is, we have that $n \equiv 1 \pmod 4$ and either $\ell = \tfrac1{2}(n - 1)$ and $m = \tfrac1{2}(5n + 1)$, or $\ell = \tfrac1{2}(n + 1)$ and $m = \tfrac1{2}(5n - 1)$.
These are cases~\ref{ite:inverse-lucas4} and \ref{ite:inverse-lucas5}.

If $m \in J_3 \setminus \{3n - \ell\}$, then \eqref{equ:Dlmn-cases} and Lemma~\ref{lem:inequalities}\ref{ite:inequality-5Fa-Fb} yield that
\begin{equation*}
    \left|\frac{L_\ell L_m}{L_n} - C_{\ell, m, n}\right| = \frac{|5F_{m-3n}F_{n - \ell}|}{L_n} < 1 - \frac1{L_n} ,
\end{equation*}
which implies that $(-1)^{\ell + n + 1} 5F_{m-3n}F_{n - \ell} = 1$, but this last equality is impossible.

If $m = 3n + \ell$, then \eqref{equ:Dlmn-cases} and Lemma~\ref{lem:products}\ref{ite:product-La-Lb} yield that
\begin{equation*}
    D_{\ell, m, n} = L_\ell L_{\ell - n} = L_{2\ell - n} + (-1)^{\ell + n} L_n .
\end{equation*}
Hence, by Lemma~\ref{lem:inequalities}\ref{ite:inequality-La}, we get that
\begin{equation*}
    \left|\frac{L_\ell L_m}{L_n} - \big(C_{\ell, m, n} + (-1)^{\ell + n})\right| = \frac{|L_{2\ell - n}|}{L_n} < 1 - \frac1{L_n} ,
\end{equation*}
which implies that $L_{2\ell - n} = 1$.
Hence, we get that $2\ell - n = 1$.
Recalling that $m = 3n + \ell$, it follows that $\ell = \tfrac1{2}(n + 1)$, $m = \tfrac1{2}(7n + 1)$, and $n$ is odd, which is case \ref{ite:inverse-lucas6}.

If $m \in J_4 \setminus \{3n + \ell, 4n\}$ and $(\ell, m) \neq (n - 1, 4n)$, then \eqref{equ:Dlmn-cases} and Lemma~\ref{lem:inequalities}\ref{ite:inequality-La-Lb} yield that
\begin{equation*}
    \left|\frac{L_\ell L_m}{L_n} - C_{\ell, m, n}\right| = \frac{|L_\ell L_{m - 4n}|}{L_n} < 1 - \frac1{L_n} ,
\end{equation*}
which implies that $L_\ell L_{m - 4n} = 1$.
Hence, we have that $\ell = 1$ and $m = 4n + 1$, which is impossible, since $m \leq 4n$.

If $\ell = n - 1$ and $m = 4n$ then, by~\eqref{equ:Dlmn-cases}, we get that
\begin{equation*}
	D_{\ell,m,n} = 2 L_{n - 1} = L_n + L_{n - 3} .
\end{equation*}
Hence, by Lemma~\ref{lem:inequalities}\ref{ite:inequality-La}, we have that
\begin{equation*}
	\left|\frac{L_\ell L_m}{L_n} - (C_{\ell,m,n} + 1)\right| = \frac{|L_{n-3}|}{L_n} < 1 - \frac1{L_n} ,
\end{equation*}
which implies that $L_{n - 3} = 1$, but this last equality is impossible.

At this point, we have proved that if~\eqref{equ:Ll-equal-to-inverse-Lm-mod-Ln} is true then the integers $\ell, m, n$ are of the form given by~\ref{ite:inverse-lucas1}--\ref{ite:inverse-lucas6}.

Vice versa, using~\eqref{equ:Dlmn-cases}, on can easily verify that if $\ell, m, n$ are of the form given by~\ref{ite:inverse-lucas1}--\ref{ite:inverse-lucas6} then $D_{\ell, m, n} = 1$.
In turn, by~\eqref{equ:def-Dlmn}, this implies that~\eqref{equ:Ll-equal-to-inverse-Lm-mod-Ln} is true.

The proof of Theorem~\ref{thm:inverse-lucas} is complete.

\end{document}